\documentclass[leqno,twoside,11pt]{amsart}

\usepackage{latexsym,esint,url, color, amssymb}
\usepackage{amsmath, amsthm, graphicx}

\theoremstyle{plain}

\newtheorem{theorem}{Theorem}
\newtheorem{corollary}[theorem]{Corollary}
\newtheorem{lemma}[theorem]{Lemma}

\theoremstyle{definition}

\def\bees{\begin{equation*}}
\def\eees{\end{equation*}}
 \def\bee{\begin{equation}}
\def\eee{\end{equation}}
 
\numberwithin{equation}{section}
\numberwithin{figure}{section}

\def\R{{\mathbb R}}

\definecolor{Green}{rgb}{0, 0.65,0}

\begin{document}

\title[Supporting spheres condition and the boundary regularity]
{Which domains have two-sided supporting unit spheres \\  at every boundary point?}
\author{Marta Lewicka and Yuval Peres}
\address{Marta Lewicka, University of Pittsburgh, Department of Mathematics, 
139 University Place, Pittsburgh, PA 15260}
\address{Yuval Peres, Microsoft Research, One Microsoft Way, WA 98052}
\email{lewicka@pitt.edu, peres@microsoft.com} 

\begin{abstract} 
We prove the quantitative equivalence of two important geometrical
conditions, pertaining to the regularity of a domain
$\Omega\subset\R^N$. These are: (i) the uniform two-sided supporting
sphere condition, and (ii)
the Lipschitz continuity of the outward unit normal vector. In
particular, the answer to  the question posed 
in our title is:  ``Those domains, whose unit normal is well
defined and has Lipschitz constant one.''  
We also offer an extension to infinitely dimensional spaces $L^p$, $p\in (1,\infty)$.
\end{abstract}

\maketitle

\section{Introduction}\label{zero}

In this note, we prove the quantitative equivalence of two geometric boundary
regularity conditions, that are often used in 
Analysis and PDEs: the uniform supporting sphere condition and
the Lipschitz continuity of the unit normal vector.
This latter condition is often referred to as the ${C}^{1,1}$ regularity
of the domain $\Omega$. In particular, the answer to  the question
in the title: ``Which domains have two-sided supporting unit spheres at every boundary point?'' 
is:  ``Those, whose outward unit normal vector is well
defined and has Lipschitz constant one.''  

\bigskip

Given $r>0$, we say that $\Omega\subset \R^N$ satisfies the {\em
  two-sided supporting $r$-sphere condition}, if: 
\begin{equation}\tag*{$\mathrm{(S}_r\mathrm{)}$}\label{S}
\left[~\mbox{\begin{minipage}{14.7cm}
For every $x_0\in\partial\Omega$ there exist $a,b\in\R^N$ such that:
$$B_r(a)\subset\Omega, \qquad B_r(b)\subset \R^N\setminus\bar\Omega \quad
\mbox{ and } \quad |x_0-a|=|x_0-b|=r.$$\end{minipage}}\right.
\end{equation}
We also say that $\Omega$ satisfies the condition of {\em $1/r$-Lipschitz continuity of
  the normal vector}, if:
\begin{equation}\tag*{$\mathrm{(L}_r\mathrm{)}$}\label{NN}
\left[~\mbox{\begin{minipage}{14.7cm}
The boundary $\partial \Omega$ is $C^{1}$ regular and the outward unit
normal $\vec n:\partial\Omega\to\R^N$ obeys:
\begin{equation*}\label{niceLipschitz}
|\vec n(x_0) - \vec n(y_0)| \leq \frac{1}{r}|x_0 - y_0|\qquad \mbox{ for all } x_0, y_0\in\partial\Omega.
\end{equation*}
\end{minipage}}\right.
\end{equation}

\smallskip

\noindent It is easy to note that for $\Omega=B_r(0)$ where \ref{S}
holds trivially, the property \ref{NN} holds as well. It turns out that
this observation can be generalized, as stated in our main result below:

\begin{theorem}\label{main}
Let $\Omega\subset\R^N$ be a domain (i.e. an open, connected
set). For any $r>0$, conditions \ref{S} and \ref{NN} are equivalent.
\end{theorem}

\smallskip

Recall that $\partial\Omega$ being $C^1$ regular, signifies that it is locally a
graph of a $C^{1}$ function, namely:
\begin{equation}\tag*{$\mathrm{(C}^1\mathrm{)}$}\label{C1}
\left[~\mbox{\begin{minipage}{14.6cm}
For every $x_0\in\partial\Omega$ there exist $\rho, h>0$ and a
rigid map $T:\mathbb{R}^N\to\R^N$ with $T(x_0) = 0$,
along with a $C^{1}$ function $\phi:\R^{N-1}\supset \bar B_\rho(0)\to
(-h, h)$ such that $\phi(0)=0,$ $\nabla \phi(0) = 0$, and the following holds. Consider
the cylinder $\mathcal{C}=B_{\rho}(0)\times (-h,h)\subset\R^N$, then:
\begin{equation*}
\begin{split}
& \mathcal{C}\cap T(\Omega) = \{(x', x_N)\in\mathcal{C}; ~ x_N<\phi(x')\},\\
&\mathcal{C}\cap T(\partial\Omega) = \{(x', x_N)\in\mathcal{C}; ~
x_N=\phi(x')\}.
\end{split}
\end{equation*}
\end{minipage}}\right.
\end{equation}
Saying that $T$ is a rigid map means that it is a composition of a
rotation and a translation: $T(x) = Ax+b$ for some $A\in SO(N)$ and
$b\in\R^N$. Further, a function $\phi$ is said to be of class $C^1$ if it is
differentiable and its gradient is continuous in the domain
where $\phi$ is defined. The geometric meaning of \ref{S} and \ref{C1}
is sketched in Figure \ref{f:supporting_balls}. We note in passing
that under condition \ref{C1}, the outward unit normal $\vec n$ is always well defined and given in the local
coordinates $\phi$ (for simplicity we assume here that $T(x) = x$) by: 
$$\vec n(x',\phi(x'))= \frac{(-\nabla \phi(x'), 1)}{\sqrt{|\nabla \phi(x')|^2+ 1}}.$$

\begin{figure}[htbp]
\centering
\includegraphics[scale=0.45]{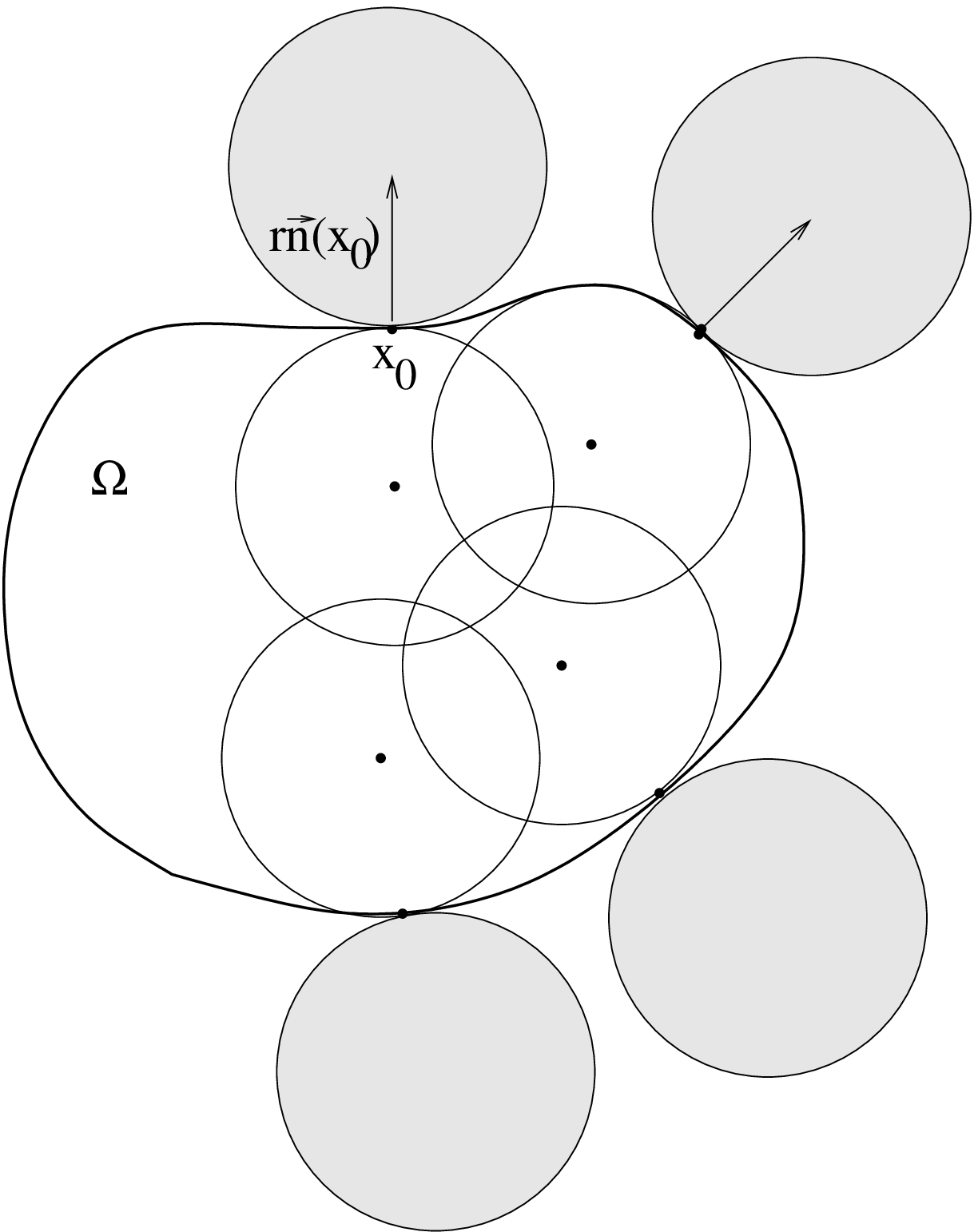}\qquad \qquad 
\qquad \includegraphics[scale=0.5]{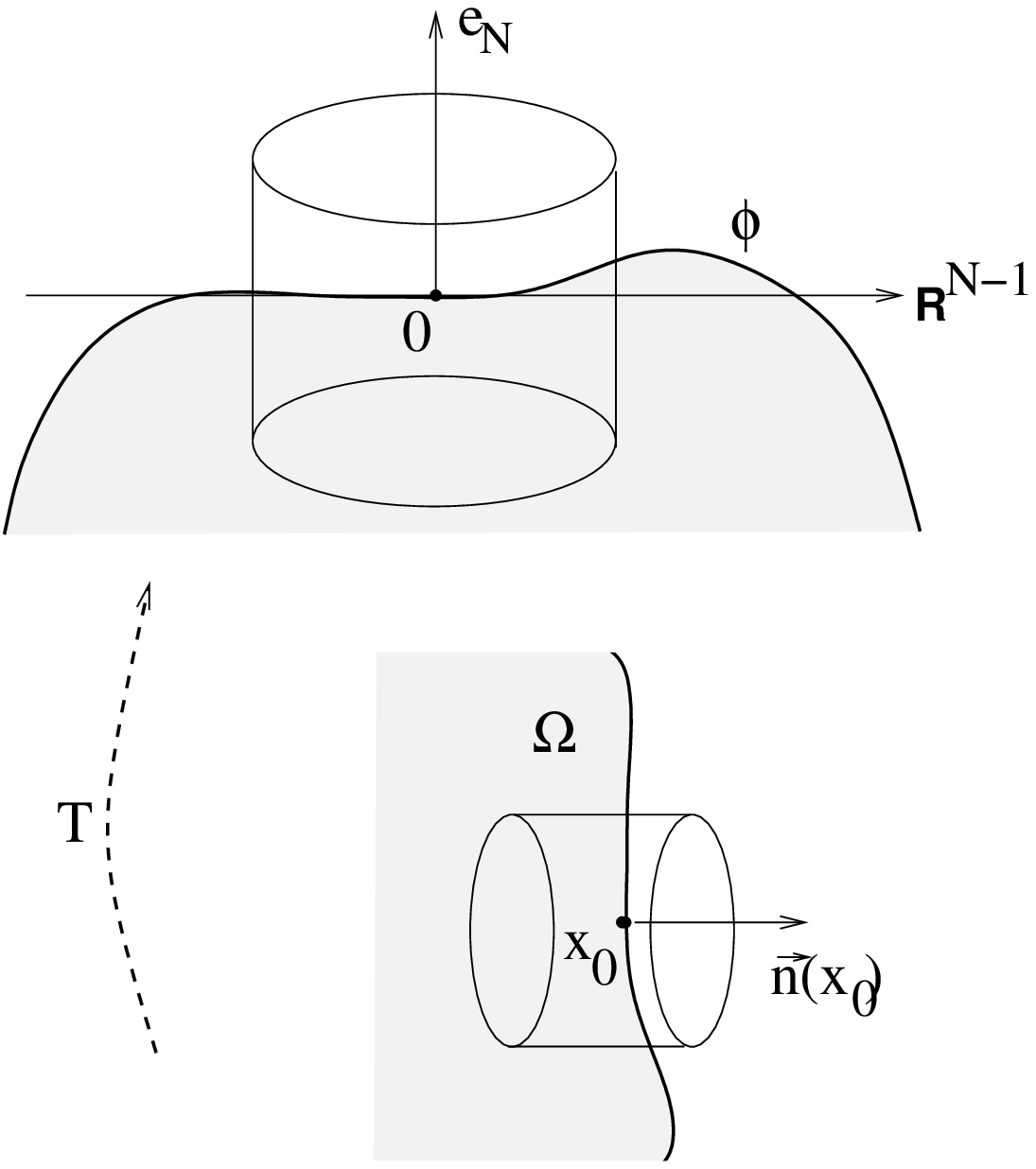}
\caption{{\small The two-sided supporting $r$-spheres condition
 \ref{S}, and the boundary regularity definitions \ref{C1} and \ref{C11}.}}
\label{f:supporting_balls}
\end{figure}



\medskip

Recall that for a domain $\Omega\subset \R^2$ whose
boundary coincides with a $C^2$ simple curve $\gamma$, parametrized by the
arc-length $s$, the radius of curvature 
at a point $\gamma(s_0)$ is: $|\frac{\mathrm{d}\vec n}{\mathrm{d}s} (s_0)|^{-1}$. 
Similarly in higher dimensions, the radius of curvature at $x_0\in\partial\Omega$ is the
reciprocal of the Lipschitz constant of $\vec n$ on a neighbourhood of
$x_0$, in the limit when the neighbourhood shrinks to the point $\{x_0\}$ itself. 
In Theorem \ref{main} we argue that this local statement persists
globally, in connection with the (global) Lipschitz constant of $\vec
n$ rather than the (locally defined) curvature. 
Indeed, the smooth ``thin neck'' set in Figure
\ref{f:thinset} has small curvature  (i.e. small local Lipschitz
constant of $\vec n$ and, equivalently, small $\|\nabla \phi\|_{C^0}$), but it 
does not allow the radius of the  internal supporting sphere to exceed a
prescribed $0<\delta\ll 1$. On the other hand, the global Lipschitz
constant of $\vec n$ in this example must be at least $\frac{|\vec n(x_0)- \vec
  n(y_0)|}{|x_0-y_0|} = \frac{2}{2\delta}$, so there is no contradiction with Theorem \ref{main}.

\begin{figure}[htbp]
\centering
\includegraphics[scale=0.55]{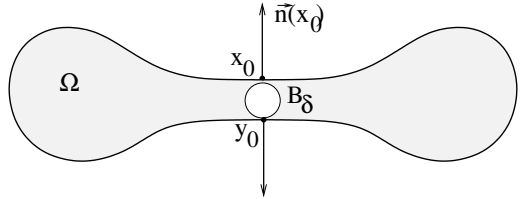}
    \caption{{\small The normal vector $\vec n$ has local Lipschitz
        constant less than $1$, but the
        maximal internal ball at $x_0$ has only a small radius $\delta\ll 1$.}}
\label{f:thinset}
\end{figure}

\medskip

The main result will be proved in Sections \ref{uno}--\ref{tre}. In
parallel, we will also deduce a consequence 
of Theorem \ref{main},  previously shown in \cite[Section 1.2]{CC} 
and \cite{PhD_SimonaBarb}, via more complex calculations:

\begin{corollary}\label{main2}
A bounded domain $\Omega\subset\R^N$ satisfies \ref{S}, for some $r>0$, if and
only if there holds:
\begin{equation}\tag*{$\mathrm{(C}^{1,1}\mathrm{)}$}\label{C11}
\left[~\mbox{\begin{minipage}{14.4cm}
The statement in \ref{C1} is valid with $C^{1,1}$ regular 
functions $\phi$, i.e. with each $\phi$ having its gradient Lipschitz continuous.
\end{minipage}}\right.
\end{equation}
\end{corollary}

In Section \ref{Lp} we will then adapt a key ingredient of the proofs of
Theorem \ref{main} and Corollary \ref{main2}, called the ``four
ball lemma'', to the setting of the Lebesgue spaces $L^p$, $p\in (1,
\infty)$.  As a result, we obtain that if a domain $\Omega\subset L^p$
satisfies \ref{S} for some $r>0$, then its outward unit normal vector is
H\"older continuous, with exponent $2/p$ for $p\geq 2$ and $p/2$ for
$p\leq 2$. In the last Section \ref{conclu} we will gather some final remarks.

\section{The four ball lemma}\label{uno}
  
The key ingredient in the proof of Theorem \ref{main} is a
geometrical lemma about four balls in $\R^N$. The balls have the same
radius $r>0$ and they come in two couples (see Figure \ref{f:4ball}),
with balls in the same couple tangent to each other. It turns out that if
we change the pairings and ensure that the two balls in each newly
formed pair are disjoint, then the directions perpendicular to the tangency planes differ at most by
the distance between the tangency points. 
A similar result remains also valid for domains $\Omega\subset
L^p(Z)$, $p\in (1,\infty)$, as will be shown in Lemma \ref{lem_4ball_Lp} in Section \ref{Lp}.

\begin{lemma}[{\bf The four ball lemma}]\label{lem_4ball}
Let $x, u, v\in \R^N$ with $|u|=|v|=r>0$. Assume that $B_r(x+u)\cap
B_r(-x-v) = \emptyset$ and $B_r(x-u)\cap B_r(-x+v)=\emptyset$. Then
$|u-v|\leq 2|x|.$
\end{lemma}
\begin{proof}
Define the vector $z=x+\frac{u+v}{2}$ and observe that:
$$\frac{u-v}{2} = (x+u)-z = (-x-v)+z \in \bar B_{|z|}(x+u)\cap \bar B_{|z|}(-x-v),$$
so that the first disjointness assumption yields:
$$r^2\leq |z|^2=\Big|x+\frac{u+v}{2}\Big|^2.$$
Exchange now $x$ with $-x$ and $u$ with $v$, and apply the second
disjointness assumption to obtain:
$$r^2\leq \Big|x-\frac{u+v}{2}\Big|^2.$$
Finally, summing the two above inequalities and using (twice) the parallelogram identity, we get:
\begin{equation}\label{2parallelo}
\begin{split}
2r^2 \leq \Big|x+\frac{u+v}{2}\Big|^2 + \Big|x-\frac{u+v}{2}\Big|^2
& = 2|x|^2 + 2\Big|\frac{u+v}{2}\Big|^2 = 2|x|^2 + \frac{1}{2}|u+v|^2
\\ & = 2|x|^2 + \frac{1}{2}\big(2|u|^2 + 2|v|^2 - |u-v|^2\big),
\end{split}
\end{equation}
which results in $|u-v|^2\leq 4|x|^2$ in view of $|u|=|v|=r$. The claim is proved.
\end{proof}

\smallskip

\begin{figure}[htbp]
\centering
\includegraphics[scale=0.6]{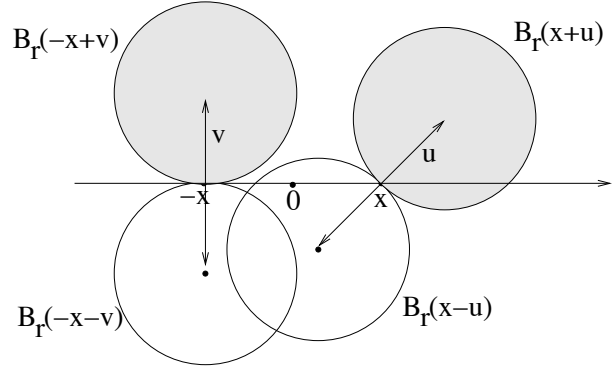}
    \caption{{\small The four balls in Lemma \ref{lem_4ball}: the ``vertical'' couples are
        tangential, the ``dia\-gonal'' couples are disjoint.}}
\label{f:4ball}
\end{figure}

\medskip

As an immediate consequence, we derive the following:

\begin{lemma}\label{SthenLipschitz}
Assume that a domain $\Omega\subset\R^N$ satisfies \ref{S}. For each $x_0\in\partial\Omega$,
define $\vec p(x_0)=\frac{b-a}{|b-a|}$. Then:
$$|\vec p(x_0) - \vec p(y_0)|\leq \frac{1}{r}|x_0-y_0|\qquad \mbox{ for all }~ x_0, y_0\in\partial\Omega.$$
\end{lemma}
\begin{proof}
We first observe that the function $\vec p:\partial\Omega\to \R^N$ is
indeed well defined, in view of \ref{S}. Applying, if needed, a rigid transformation
that maps a given $x_0, y_0\in\partial\Omega$
to some symmetric points $x$ and $-x$, we now use Lemma \ref{lem_4ball} to
$u=r\vec p(x_0)$, $v= r\vec p(y_0)$ and conclude that $|r\vec p(x_0) -
r\vec p(y_0)|\leq |x_0-y_0|$. This proves the claim  (see Figure \ref{f:4ball_boundary}).
\end{proof}

\begin{figure}[htbp]
\centering
\includegraphics[scale=0.6]{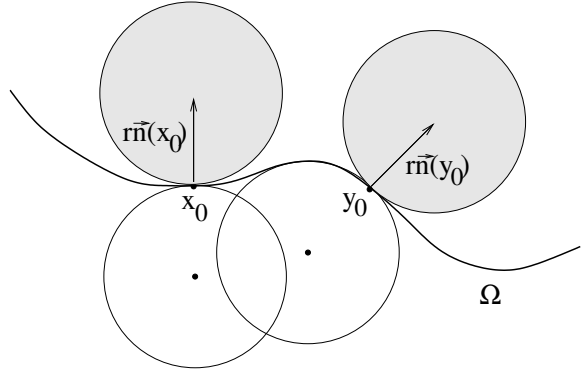}
    \caption{{\small The supporting balls at two boundary points $x_0$
        and $y_0$. The ``external'' supporting balls are shaded.}}
\label{f:4ball_boundary}
\end{figure}

\section{A proof of Theorem \ref{main} and Corollary \ref{main2}: 
\ref{S} implies \ref{NN} and \ref{C11}}\label{due} 

It is geometrically clear that, given \ref{S}, the normal vector $\vec
n(x_0)$ must coincide with the, previously introduced, normalized shift between the two
supporting balls at $x_0\in\partial\Omega$:
$$\vec p(x_0) = \frac{b-a}{|b-a|},$$
and thus the statement in Lemma \ref{SthenLipschitz} is essentially
that of \ref{S} implies \ref{NN} in Theorem \ref{main}. Below, we verify this statement formally,
together with the parallel implication in Corollary \ref{main2}.

\begin{lemma}\label{2ndimplication}
Assume that a domain $\Omega\subset\R^N$ satisfies \ref{S}. 
Then $\vec p=\vec n$ on $\partial\Omega$ and \ref{C11} holds.
\end{lemma}
\begin{proof}
By applying some rigid transformation $T$, we may without loss of
generality assume that $x_0=0$ and $\vec p(0) = e_N$. We now claim that
choosing $h\gg\delta>0$ sufficiently small, for each $x'\in
B_\delta(0)\subset\R^{N-1}$ there exists a unique $\phi(x')\in\R$, satisfying:
$$(x', \phi(x'))\in\partial\Omega\cap \mathcal{C}, \qquad \mbox{where
} ~ \mathcal{C} = B_\delta(0)\times (-h, h).$$
Indeed, consider the two supporting balls $B_r(-re_N)$ and $B_r(re_N)$
at $0$; when $\delta<\frac{r}{2}$ then the line $x'+\mathbb{R}e_N$
intersects both of them, so for a small fixed $h$ and
sufficiently small $\delta$, there is at least one
$\phi(x')$ with the indicated property. To prove uniqueness, take
$\bar\delta<\frac{r}{2}$ small enough for Lemma \ref{SthenLipschitz} to guarantee
that the angle between $\vec p(y_0)$ and $e_N$ 
is less than $\frac{\pi}{4}$ for all $y_0\in\partial\Omega\cap
B_{\bar\delta}(0)$. Then, if $y_0, \bar y_0\in \partial\Omega\cap B_{\bar\delta}(0)$
satisfy $\bar y_0-y_0=t e_N $ for some $t\in (0,r)$, it follows that $\bar y_0$ belongs
to the external supporting ball at $y_0$, contradicting the fact that $\bar
y_0\in\partial\Omega$. It now suffices to ensure that
$\mathcal{C}\subset B_{\bar\delta}(0)$ by taking $h,\delta\ll 1$.

We have thus defined the function $\phi$ whose graph
locally coincides with $\partial\Omega$. By similar arguments as
above, \ref{S} implies that $\phi$ must be continuous,
differentiable and also:
$$\vec{n}(x', \phi(x')) = \frac{(-\nabla
  \phi(x'),1)}{\sqrt{|\nabla\phi(x')|^2+1}}= \vec p(x', \phi(x')) \qquad \mbox{ for
  all } x'\in\partial B_\delta(0).$$
The above formula and Lemma \ref{SthenLipschitz}
give that, for every $j=1,\ldots,N-1$, the function: 
$$\partial_j\phi(x') = \frac{\langle \vec p(x', \phi(x')),
  e_j\rangle}{\langle \vec p(x', \phi(x')), e_N\rangle}$$
is continuous on $B_\delta(0)$ as the quotient of two continuous
functions, whose denominator is bounded away from $0$. It now follows that
functions involved in the above quotient are actually Lipschitz, so
$\nabla \phi$ must be Lipschitz as well. This ends the proof of
$\phi\in C^{1,1}$ and the proof of the Lemma.
\end{proof}

\section{A proof of Theorem \ref{main} and Corollary \ref{main2}: 
both \ref{NN} and \ref{C11} imply \ref{S} }\label{tre}

We now complete the proofs of our main results.

\begin{lemma}\label{before}
Let $\Omega\subset\R^N$ be a domain. At a given
$x_0=0\in\partial\Omega$, let $\phi$ represent the local coordinates
of $\partial\Omega$ as in \ref{C1}, and with $T=id$, so that $\phi(0) = 0$
and $\nabla \phi(0)=0$. Then:
\begin{equation}\label{pomoc3}
|\phi(x')| \leq \frac{\max_{\bar B_{|x'|}(0)}|\nabla
  \phi|^2+1}{2r}|x'|^2\qquad \mbox{for all } x'\in B_\rho(0)
\end{equation}
is valid, in the following two cases:
\begin{itemize}
\item[(i)] condition \ref{NN} holds,
\item[(ii)] condition \ref{C11} holds and 
  $\frac{1}{r}$ bounds the Lipschitz constant of $\nabla\phi$ from above.
\end{itemize}
Consequently, in both cases, $x_0$ has the supporting balls 
$B_{\delta_0}(0,\delta_0)$, $B_{\delta_0}(0,-\delta_0)$ with the radius:
$$ \delta_0 = \min\Big\{\frac{r}{\max_{\bar B_{\rho}(0)}|\nabla
  \phi|^2+1}, \rho, \frac{h}{2}\Big\}.$$
\end{lemma}
\begin{proof}
Since $\vec n(x_0) = e_N$, condition \ref{NN} implies:
$$|\vec n(x', \phi(x')) - e_N|^2\leq \frac{1}{r^2} \big(|x'|^2 +
|\phi(x')|^2\big). $$ 
Consequently, in case (i), for all $x'\in B_\rho(0)$ there holds:
\begin{equation}\label{pomoc2}
|\nabla \phi(x')|^2 \leq \frac{|\nabla \phi(x')|^2+1}{r^2}\big(|x'|^2
+ |\phi(x')|^2\big) \leq \frac{\big(\max_{\bar B_{|x'|}(0)}|\nabla \phi|^2+1\big)^2}{r^2}|x'|^2,
\end{equation}
where we have used the fact that $\phi(0) = 0$ and $\nabla \phi(0)=0$ to get:
\begin{equation}\label{num}
|\phi(x')| = \Big| \int_0^1\langle \nabla \phi(tx'),
x'\rangle~\mbox{d}t\Big|\leq |x'| \int_0^1 |\nabla \phi(tx')|~\mbox{d}t.
\end{equation}
Clearly, in case (ii) we have: $|\nabla \phi(x')|\leq \frac{1}{r}|x'|$,
so (\ref{pomoc2}) holds then as well.
Applying now (\ref{pomoc2}) in the right hand side of (\ref{num}), we derive
(\ref{pomoc3}). 

To prove the final statement, we first note from
(\ref{pomoc3}) that the graph of $\phi$ is contained
between two parabolas $x'\mapsto \pm \frac{1}{2\delta}|x'|^2$, where 
$ \delta = \frac{r}{\max_{\bar B_{\rho}(0)}|\nabla \phi|^2+1}$. It
hence easily follows that:
\begin{equation}\label{simple}
|(x', \phi(x')) - (0,\pm\delta)|^2 
\geq {2\delta}|\phi(x')| +  |\phi(x') \mp \delta|^2 \geq  
2\delta|\phi(x')|+ \big( \delta^2 - 2\delta|\phi(x')|  \big) = \delta^2.
\end{equation}
Decreasing the radius, if necessary, to the indicated value $\delta_0\leq \delta$, we obtain that
the balls $B_{\delta_0}(0,\delta_0)$ and $B_{\delta_0}(0,-\delta_0)$ are supporting at $x_0$.
\end{proof}

\medskip

We readily deduce:

\begin{lemma}\label{corfinish}
Let $\Omega\subset\R^N$ be a bounded domain, satisfying \ref{C11}. Then \ref{S} holds with some $r>0$.
\end{lemma}
\begin{proof}
For each $x_0\in\partial\Omega$, denote by $\mathcal{C}(x_0)=B_\rho(0)\times (-h,h)\subset\R^N$
the cylinder with radius $\rho=\rho(x_0)$ and height $h=h(x_0)$, to
which the definition \ref{C1} applies. We first observe 
that fixing a sufficiently small $\eta=\eta(x_0)$, every point
$y_0\in\partial\Omega\cap B_\eta(x_0)$ has the property that 
the halved cylinder ${\mathcal{C}}_{1/2}(x_0)=B_{\frac{\rho}{2}}(0)\times
(-\frac{h}{2}, \frac{h}{2})$ can also be taken as its corresponding $\mathcal{C}$.
In particular, since $(T_{y_0})^{-1}\big(\mathcal{C}_{1/2}(x_0)\big)
\subset (T_{x_0})^{-1}\big(\mathcal{C}(x_0)\big)$ and the graph of
$\phi_{y_0}$ is a rigid motion of a part of the graph of
$\phi_{x_0}$, it follows that the radius $\delta_0$ of the
supporting balls guaranteed in Lemma \ref{before}, is the same for
every $y_0\in \partial\Omega\cap B_\eta(x_0)$.
By compactness, $\partial\Omega$ may be covered by finitely many
balls in the family $\{B_\eta (x_0)\}_{x_0\in\partial\Omega}$. Taking
the smallest of such constructed radii $\{\delta=\delta_0(x_0)\}$ proves \ref{S}. 
\end{proof}

\medskip

The following final argument completes the proof of Theorem \ref{main}:

\begin{lemma}\label{1stimplication}
Assume that a domain $\Omega\subset\R^N$ satisfies \ref{NN}. Then \ref{S} holds.
\end{lemma}
\begin{proof}
For a boundary point $x_0\in\partial\Omega$, let $\phi$ be as described in condition \ref{C1}, where
without loss of generality we take $T=id$, so that Lemma \ref{before}
may be used. We argue by contradiction. If $B_r(re_N)$ was
not supporting, then $B_r(re_N)\cap \Omega\neq \emptyset$ and further:
$$\bar r\doteq \inf\big\{\delta\geq \delta_0; ~ \bar B_\delta(0,
\delta)\cap\partial\Omega\neq \{0\}\big\} < r.$$
Take a sequence of radii:  $\delta_n\searrow \bar r$ as
$n\to\infty$, and a sequence of points: $y_n\in\bar B_{\delta_n}(
\delta_ne_N)\cap\big(\partial\Omega \setminus \{0\}\big)$.
Without loss of generality, $y_n\to y_0\in \bar B_{\bar r}(\bar re_N)
\cap \partial\Omega$. By the minimality of $\bar r$, there must be:
$$y_0\in \partial B_{\bar r}(\bar re_N)\cap \partial\Omega.$$
Since the inward normal $-\vec n(y_0) $ coincides with the normal to
$\partial B_{\bar r}(\bar re_N)$ at $y_0$ and since the same is valid at
$x_0=0$, we get the following equality, which
in virtue of \ref{S} implies:
$$\Big|\frac{1}{\bar r}y_0 - \frac{1}{\bar r}x_0\Big| = \big|\vec
n(x_0) - \vec n(y_0) \big| \leq \frac{1}{r} |x_0-y_0|. $$
This results in the contradictory statement $\bar r\geq r$, provided
that we show $y_0\neq x_0$.
To this end, it suffices to argue that $\{y_n\}$ must be bounded away from
$0$. Let $\sigma>0$ satisfy:
$$\big(\sup_{\bar B_\sigma(0)}|\nabla
  \phi|^2+1\big)\sigma^2\leq 2r \bar r \leq 2r\delta_n \qquad \mbox{for all } n.$$ 
Using (\ref{pomoc3}) we refine the bound in (\ref{simple}) for $x'\in B_\sigma(0)$, to:
\begin{equation*}
\begin{split}
|(x', \phi(x')) - \delta_ne_N|^2 & = |x'|^2 + |\delta_n-\phi(x')|^2\geq
|x'|^2 +\Big|\delta_n- \frac{\sup_{\bar B_\sigma(0)}|\nabla
  \phi|^2+1}{2r}|x'|^2\Big|^2 \\ & \geq |x'|^2 + \delta_n^2 -
  \frac{\delta_n}{r}\big(\sup_{\bar B_\sigma(0)}|\nabla \phi|^2+1\big) |x'|^2. 
\end{split}
\end{equation*}
Note that for all large $n$, the ratio $\frac{\delta_n}{r}$ is smaller
than and bounded away from $1$. It follows that for sufficiently small
$\sigma>0$ there also holds: 
$\frac{\delta_n}{r}\big(\sup_{\bar B_\sigma(0)}|\nabla \phi|^2+1\big)< 1 $ for all $n$. Consequently:
$$|(x', \phi(x')) - \delta_ne_N|^2\geq \delta_n^2\qquad \mbox{ for
  all } x'\in B_\sigma(0)\quad \mbox{and all } n,$$
where equality only takes place at $x'=0$. This implies that $|y_n'|\geq \sigma$
and hence $|y_0|\geq \sigma$ as well, as claimed.
We conclude that the ball $B_r(re_N)$ is
external supporting. The proof that $B_r(-re_N)$ is an internal
supporting ball follows in the same manner.
\end{proof}

\section{An extension of the four ball lemma to the Lebesgue spaces}\label{Lp}

We now discuss an extension of our results to infinitely dimensional Banach spaces
$L^p=L^p(Z)$, where $Z$ is an arbitrary measure space. We denote by $\|u\|$ the $L^p$ norm
of a given $u\in L^p$. Observe that the ``four ball lemma'' (Lemma \ref{lem_4ball})
remains valid in any inner product space, i.e. in a space where one can use the
parallelogram identity. Thus, it is valid in
$L^2$, whereas for $p\neq 2$ we still get the following:

\begin{lemma}[{\bf The four ball lemma in $L^p$}]\label{lem_4ball_Lp}
Let $x, u, v\in L^p$ with $\|u\|=\|v\|=r>0$. Assume that $B_r(x+u)\cap
B_r(-x-v) = \emptyset$ and $B_r(x-u)\cap B_r(-x+v)=\emptyset$. Then:
\begin{equation}\label{stima}
\begin{split}
& \|u-v\|^p\leq 2^{p-1}p(p-1) r^{p-2}\|x\|^2 \qquad  \mbox{ for } \; p\in [2,\infty)\\
& \|u-v\|^2\leq \frac{8}{p(p-1)} r^{2-p}\|x\|^p \;\;\quad\qquad  \mbox{ for } \; p\in (1,2].\\
\end{split}
\end{equation}
\end{lemma}
\begin{proof}
By the same reasoning as in the first part of the proof of Lemma \ref{lem_4ball}, we obtain:
$$r\leq \Big\|x+\frac{u+v}{2}\Big\| \qquad \mbox{and} \qquad r\leq \Big\|x-\frac{u-v}{2}\Big\|.$$
We will now replace the parallelogram identity argument in Lemma \ref{lem_4ball}, with a
somewhat more involved estimate, derived separately in the two cases indicated in (\ref{stima}).

\smallskip

When $p\geq 2$, we use first the 2-uniform smoothness inequality
\cite[Proposition 3]{BCL}, followed by the first Clarkson's inequality
\cite[page 95]{Brezis} and recall that $\|u\|=\|v\|=r$, to get:
\begin{equation}\label{st}
\begin{split}
2r^2 &\leq \Big\|x+\frac{u+v}{2}\Big\|^2 + \Big\|x-\frac{u+v}{2}\Big\|^2
 \leq 2(p-1)\|x\|^2 + 2\Big\|\frac{u+v}{2}\Big\|^2\\ &  \leq 2(p-1)\|x\|^2 
+ 2\Big(\frac{1}{2}\|u\|^p + \frac{1}{2}\|v\|^p -\Big\|\frac{u+v}{2}\Big\|^p\Big)^{2/p} 
\\ & = 2(p-1)\|x\|^2 + 2\Big(r^p -\Big\|\frac{u+v}{2}\Big\|^p\Big)^{2/p}. 
\end{split}
\end{equation}
Further, concavity of the function $a\mapsto a^{2/p}$ implies that
$(a-b)^{2/p}\leq a^{2/p} - \frac{2}{p}a^{(2/p)-1}b$ whenever $a,b, a-b\geq 0$.
Applying this bound to $a=r^p$ and $b=\|\frac{u-v}{2}\|^p$ in (\ref{st}), yields:
\begin{equation*}
2r^2 \leq  2(p-1)\|x\|^2 + 2\Big(r^2 -\frac{2}{p}r^{2-p} \Big\|\frac{u-v}{2}\Big\|^p\Big),
\end{equation*}
which directly results in (\ref{stima}) in case $p\in [2,\infty)$.

\smallskip

For $p\leq 2$, we denote the conjugate exponent by $q=\frac{p}{p-1}$
and use the second Clarkson's inequality \cite[page 97]{Brezis},
followed by the 2-uniform smoothness inequality \cite[Proposition 3]{BCL} to get:
\begin{equation*}
\begin{split}
2r^q &\leq \Big\|x+\frac{u+v}{2}\Big\|^q + \Big\|x-\frac{u+v}{2}\Big\|^q
 \leq 2\Big(\|x\|^p + \Big\|\frac{u+v}{2}\Big\|^p\Big)^{q/p} = 
2\Big(\|x\|^p + \Big(\Big\|\frac{u+v}{2}\Big\|^2\Big)^{p/2}\Big)^{q/p}
\\ &  \leq 2\Big(\|x\|^p 
+ \Big(\frac{1}{2}\|u\|^2 + \frac{1}{2}\|v\|^2 -(p-1)\Big\|\frac{u-v}{2}\Big\|^2\Big)^{p/2}\Big)^{q/p} 
\\ & = 2\Big(\|x\|^p 
+ \Big(r^2 -(p-1)\Big\|\frac{u-v}{2}\Big\|^2\Big)^{p/2}\Big)^{q/p}.
\end{split}
\end{equation*}
As before, concavity of the function $a\mapsto a^{p/2}$ applied to
$a=r^2$ and $b=(p-1)\|\frac{u-v}{2}\|^2$, yields:
\begin{equation*}
2r^q \leq  2\Big(|x|^p + r^p -\frac{(p-1)p}{2}r^{p-2} \Big\|\frac{u-v}{2}\Big\|^2\Big)^{q/p}.
\end{equation*}
Dividing both sides of the above inequality by $2r^q$ we obtain:
$$1\leq \frac{|x|^p}{r^p} + 1 -\frac{(p-1)p}{2}r^{-2} \Big\|\frac{u-v}{2}\Big\|^2$$
which directly gives (\ref{stima}) in case $p\in (1,2]$.
\end{proof}

\medskip

Similarly to Lemma \ref{SthenLipschitz}, there follows H\"older's
regularity of the unit normal to domains satisfying the two-sided uniform
supporting sphere condition:

\begin{corollary}
Let $\Omega\subset L^p$ be a domain satisfying \ref{S}. Then the
outward unit normal vector, defined for each
$x_0\in\partial\Omega$ by $\vec p(x_0) = \frac{b-a}{\|b-a\|}$, is H\"older
continuous with exponent: $\min\{2/p, p/2\}$. Namely:
\begin{equation*}
\|\vec p(x_0)- \vec p(y_0)\|\leq C_p\cdot \left\{\begin{array}{cc} \displaystyle{\frac{1}{r^{2-2/p}}
    \|x_0-y_0\|^{2/p}} & \mbox{ for } p\in [2,\infty)\vspace{1mm}\\ \displaystyle{
\frac{1}{r^{p/2}} \|x_0-y_0\|^{p/2}} & \mbox{ for } p\in (1, 2],
\end{array}\right. 
\end{equation*}
holds for all $x_0, y_0\in\partial\Omega$, with a constant $C_p>0$ that
depends only on $p\in (1,\infty)$.
\end{corollary}

\section{Concluding remarks}\label{conclu}

Theorem \ref{main} is consistent with the fact that if
$\phi\in C^2$, $\phi(0)=0$, $\nabla \phi(0)=0$, then the surface patch given by the graph of $\phi$
has, at $x_0=0$, an external supporting sphere of radius $r$, for any $r>0$ satisfying
$\frac{1}{r}\geq \lambda_{max}$, where
$\lambda_{max}$ is the maximal eigenvalue of the symmetric matrix
$\nabla^2\phi(0)$. More generally, the two-sided supporting sphere
radius $r$ at a point $x_0$ of a $C^2$ surface $S$ is the inverse of the largest (in absolute
value) eigenvalue of the second fundamental form of $S$ at $x_0$. 

Observe also that condition \ref{S} is more restrictive than the merely
assuming existence of the two-sided supporting spheres at each
boundary point, even for a bounded domain (see Figure \ref{f:ball_tail}).

\medskip

\begin{figure}[htbp]
\centering
\includegraphics[scale=0.6]{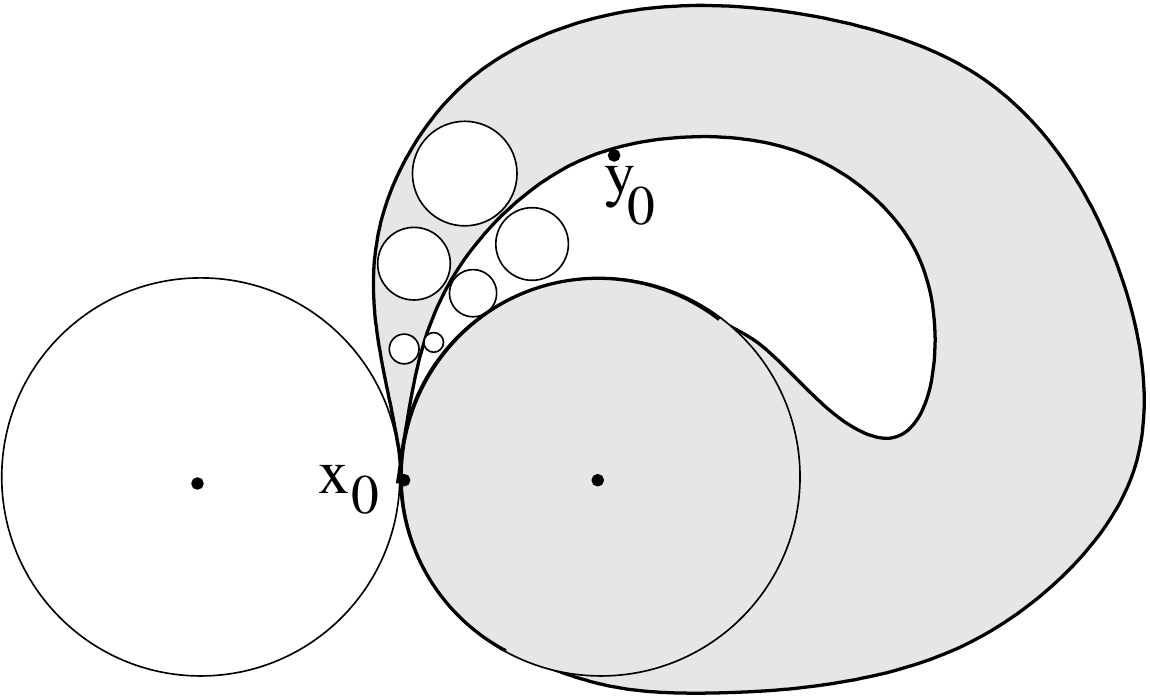}
\caption{{\small A bounded domain $\Omega$ satisfying the nonuniform
    two-sided supporting spheres condition. The maximal radii converge
    to $0$ as $y_0\to x_0$.}}
\label{f:ball_tail}
\end{figure}

\medskip

For completeness, we mention that an entirely equivalent notion of
$C^{1,1}$ regularity of $\partial\Omega$ is:
 \begin{equation}\tag*{$\mathrm{(C}^{1,1}_{\Phi}\mathrm{)}$}\label{C11'}
\left[~\mbox{\begin{minipage}{14.4cm}
For every $x_0\in\partial\Omega$ there exist $\rho, h>0$
and a $C^{1,1}$ diffeomorphism $\Phi:\mathcal{C}\to U$ between
the cylinder $\mathcal{C}=B_{\rho}(0)\times (-h,h)\subset\R^N$ and an open
neighbourhood $U\subset\R^N$ of $x_0$, such that:
\begin{equation*}
\begin{split}
& \Phi\big((x', x_N)\in\mathcal{C}; ~ x_N<0\big) = U\cap\Omega,\\
&\Phi\big((x', x_N)\in\mathcal{C}; ~ x_N=0\big) = U\cap\partial\Omega,
\end{split}
\end{equation*}
\end{minipage}}\right.
\end{equation}
We recall that $\Phi$ is a $C^{1,1}$ diffeomorphism when it is
invertible and when both $\Phi$ and $\Phi^{-1}$ have regularity
$C^{1,1}$. 
Condition \ref{C11'} is at the heart of the technique of ``straightening
the boundary''. This technique is familiar to analysts and relies 
on reducing an argument (e.g. constructing an extension operator,
deriving estimates on a solution to some 
PDEs, etc) needed in a proximity of a boundary point of a
domain, to the simpler case of flat boundary to the
half-space. The two cases are then related via the diffeomorphism
$\Phi$ with controlled derivatives (sometimes more than just
one derivative is needed!). On the other hand, it is often  more straightforward
to deal with the geometric condition \ref{S} rather than with
requirements on $\phi$ in \ref{C11} or $\Phi$ in \ref{C11'}.

\end{document}